\documentclass[brochure,english,12pt]{bourbaki}

\usepackage{lmodern,amssymb,bm}

\usepackage[utf8]{inputenc}
\usepackage{amssymb}
\usepackage{microtype}
\usepackage{url}

\date{Mars 2019}
\bbkannee{71\textsuperscript{e} ann\'ee, 2018--2019}  
\bbknumero{1159}                                      

\newtheorem{principle}[subsection]{Principle}
\newtheorem{lemma}[subsection]{Lemma}
\newtheorem{theorem}[subsection]{Theorem}
\newtheorem{problem}[subsection]{Problem}
\newtheorem{conjecture}[subsection]{Conjecture}
\newtheorem{proposition}[subsection]{Proposition}

\renewcommand{\leq}{\leqslant}
\renewcommand{\geq}{\geqslant}

\newcommand{\E}{\mathbb{E}}
\newcommand{\R}{\mathbb{R}}
\newcommand{\N}{\mathbb{N}}
\newcommand{\p}{\mathbb{P}}
\newcommand{\meas}{\mathrm{meas}}

\title{The Riemann zeta function in short intervals}
\subtitle{after Najnudel, and Arguin, Belius, Bourgade, Radziwi\l\l \, and Soundararajan}
\author{Adam J Harper}
\address{Mathematics Institute, Zeeman Building, \\
University of Warwick, Coventry, \\
CV4 7AL, England}
\email{A.Harper@warwick.ac.uk}

\begin{document}

\maketitle

\section*{Introduction}
The Riemann zeta function $\zeta(s)$ is one of the most important and fascinating functions in mathematics. When the complex number $s$ has $\Re(s) > 1$, we have
$$ \zeta(s) = \sum_{n=1}^{\infty} \frac{1}{n^s} = \prod_{\text{primes} \; p} \Bigl(1 - \frac{1}{p^s} \Bigr)^{-1} , $$
and already from these equivalent expressions we see some of the key themes that dominate the study of $\zeta(s)$. 

Firstly, since $\zeta(s)$ is given by a {\em Dirichlet series} over all natural numbers $n$, without any difficult number theoretic coefficients, we can hope to use general analytic methods to obtain information about $\zeta(s)$. For example, one could hope to approximate $\sum_{n=1}^{\infty} \frac{1}{n^s}$ or its partial sums by an integral. In this way, one can extend the definition of $\zeta(s)$ to all $\Re(s) > 0$, and with more work to the entire complex plane. It turns out that this analytic continuation of $\zeta(s)$ is meromorphic, with only a simple pole at $s=1$. Furthermore, when $\Re(s) > 0$ the zeta function is the sum $\sum_{n \leq X} \frac{1}{n^s}$ plus some easily understood other terms, for suitable $X = X(s)$.

Secondly, since $\zeta(s)$ is given by an {\em Euler product} over all primes $p$, we can hope to use results about the zeta function to draw conclusions about the distribution of primes. One can also go in the reverse direction, and hope to put in information about the primes to deduce things about the zeta function (from which, perhaps, we will later deduce other number theoretic information that we didn't have before). In this article we will discuss various results of this nature.

Thirdly, note that the Euler product is absolutely convergent when $\Re(s) > 1$, and none of the individual factors $(1 - \frac{1}{p^s})^{-1}$ vanish, so we have $\zeta(s) \neq 0$ when $\Re(s) > 1$. It is well known that the zeros of the zeta function encode number theoretic information, and here one can glimpse why--- if one knows that $\zeta(s)$ doesn't vanish in a certain part of the complex plane, this suggests that something like the Euler product formula persists there, which implies something about the regularity of the primes. Again there is a kind of duality, since not only does the non-vanishing of zeta imply results about primes and products, but our methods for proving non-vanishing tend to involve establishing the influence of some kind of product formula in the region under study.

\medskip
The most interesting subset of the complex plane on which to study $\zeta(s)$ is the {\em critical line} $\Re(s) = 1/2$. Thus the {\em Riemann Hypothesis} (RH), possibly the most famous unsolved problem in pure mathematics, conjectures that if $\zeta(s)=0$ then either $s=-2,-4,-6,\dots$ (the so-called trivial zeros), or else $\Re(s) = 1/2$. This is known to be equivalent to the estimate
$$ \Bigl|\#\{p \leq x : p \; \text{prime}\} - \int_{2}^{x} \frac{dt}{\log t}\Bigr| \ll x^{1/2} \log x $$
for the counting function of the primes (RH holds if and only if this estimate
holds for all large $x$). For any fixed $\sigma > 1/2$, it is believed (and to
some extent known) that the values taken by $\zeta(\sigma+it)$ have a rather
simple behaviour: for example, $\zeta(\sigma+it)$ can only attain unusually
large values as a result of ``conspiracies'' in the behaviour of $p^{-it}$ for
small primes $p$. As we shall discuss extensively later, the situation on the
critical line is very different. All the appearances of 1/2 here reflect the
fact that in a random process, the typical size of fluctuations is like the
square root of the variance. The extent to which $\zeta(s)$ behaves like various random objects, especially random objects related to the Euler product, is another key theme that we are going to explore.

\medskip
Our goal in this paper is to survey some recent work on the behaviour of $\zeta(1/2+it)$ in short intervals of $t$. In particular, we shall describe a conjecture of Fyodorov, Hiary and Keating~\cite{fyodhiarykeat, fyodkeat} about the size of $\max_{0 \leq h \leq 1} |\zeta(1/2 + it + ih)|$ as $t$ varies, and we shall explain some results that have been proved in the direction of this conjecture by Najnudel~\cite{najnudel} and by Arguin--Belius--Bourgade--Radziwi\l\l--Soundararajan~\cite{abbrs}.

This paper is organised as follows. Firstly we shall set out some Basic Principles that will guide our thinking and arguments about the zeta function. Then, to illustrate the use of these principles and to compare with the later case of $\max_{0 \leq h \leq 1} |\zeta(1/2 + it + ih)|$, we shall describe what is known about the value distribution of $\zeta(1/2+it)$ (without any maximum) and what is known about the ``long range'' maximum $\max_{T \leq t \leq 2T} |\zeta(1/2 + it)|$. Next we shall introduce and motivate the conjecture of Fyodorov--Hiary--Keating, primarily using our Basic Principles rather than the random matrix theory/statistical mechanics arguments originally considered by those authors (although we shall mention those briefly). And then we shall discuss the statements and proofs of the results of Najnudel and of Arguin--Belius--Bourgade--Radziwi\l\l--Soundararajan, again seeing how these correspond to very nice implementations of those principles.

\medskip
{\em Notation.} We shall use various notation, of a fairly standard kind, to aid in the description of estimates and limiting processes. We write $f(x) = O(g(x))$, or $f(x) \ll g(x)$, to mean that there exists some constant $C$ such that $|f(x)| \leq Cg(x)$ for all $x$ of interest (the range of $x$ should always either be clear, or specified explicitly). We write $f(x) = \Theta(g(x))$, or $f(x) \asymp g(x)$, to mean that there exist constants $0 < c < C$ such that $cg(x) \leq f(x) \leq Cg(x)$ for all $x$ of interest. At some points we will give rough or heuristic descriptions of arguments, and in these we will use notation such as $\lesssim, \approx$. We do not try to assign a precise meaning to these symbols--- they mean that one quantity is smaller than another, or roughly the same size as another, up to terms that turn out to be unimportant in the rigorous implementation of the arguments.

Much of this paper will involve the discussion of probabilistic issues. We will write $\p$ to denote a probability measure, and $\E$ to denote expectation (i.e. averaging, or more formally integration) with respect to the measure $\p$.

\section{Basic principles}
One can build up a great deal of understanding of the zeta function beginning from the following idea, which we first state in a heuristic way.
\begin{principle}\label{bprin1}
As $t$ varies, the numbers $(p^{-it})_{p \; \text{prime}}$ ``behave like'' a sequence of independent random variables, each distributed uniformly on the complex unit circle.
\end{principle}

It is clear that for any given $p$, as $t$ varies over an interval the quantity $p^{-it} = e^{-it\log p}$ rotates around the complex unit circle at ``speed'' $\log p$, behaving like a uniform random variable. Thus the interesting assertion in Principle~\ref{bprin1} is that we should think of the~$p^{-it}$ as being independent. That is because the primes are multiplicatively independent, or equivalently the speeds $\log p$ are linearly independent over the rationals. Both of these statements are just ways of expressing the uniqueness of prime factorisation. So as each of the $p^{-it}$ rotate around, there are no fixed relations between any combinations of them and so, heuristically, they shouldn't ``see one another's behaviour'' too much (unlike if we considered $2^{-it}, 3^{-it}, 6^{-it}$, say, for which always $6^{-it} = 2^{-it} 3^{-it}$).

What rigorous statements can we make that would correspond to Principle~\ref{bprin1}? The following result, although easily proved, turns out to be a very powerful tool.
\begin{lemma}\label{bcorrlem}
Let $T \geq 1$, and let $p_{1}, \dots, p_{k}, p_{k+1}, \dots, p_{\ell}$ be any primes (not necessarily distinct). Let $(X_p)_{p \; \text{prime}}$ be a sequence of independent random variables, each distributed uniformly on the complex unit circle. Then
$$ \frac{1}{T} \int_{T}^{2T} \prod_{j=1}^{k} p_{j}^{-it} \overline{\biggl(\prod_{j=k+1}^{\ell} p_{j}^{-it} \biggr)} dt = \E \prod_{j=1}^{k} X_{p_{j}} \prod_{j=k+1}^{\ell} \overline{X_{p_{j}}} + O\biggl(\frac{\min\{\prod_{j=1}^{k} p_{j}, \prod_{j=k+1}^{\ell} p_{j}\}}{T} \biggr) . $$ 
\end{lemma}

\begin{proof}[Proof of Lemma~\ref{bcorrlem}]
We can rewrite the integral on the left as
$$ \frac{1}{T} \int_{T}^{2T} \exp\Bigl\{-it\Bigl(\sum_{j=1}^{k} \log p_{j} - \sum_{j=k+1}^{\ell} \log p_{j}\Bigr) \Bigr\} dt . $$
So if $\sum_{j=1}^{k} \log p_{j} = \sum_{j=k+1}^{\ell} \log p_{j}$ then the
integral is exactly~$1$. And since this is equivalent (by uniqueness of prime
factorisation) to saying that the $(p_j)_{j=k+1}^{\ell}$ are just some
reordering, with the same multiplicities, of the $(p_j)_{j=1}^{k}$, we see
that,  in this case as well, 
$\E \prod_{j=1}^{k} X_{p_{j}} \prod_{j=k+1}^{\ell} \overline{X_{p_{j}}} = 1$ since every $X_p$ is paired with a conjugate copy.

If $\sum_{j=1}^{k} \log p_{j} \neq \sum_{j=k+1}^{\ell} \log p_{j}$, then on the right some $X_p$ is not paired with a conjugate copy, so by independence and symmetry of the distributions of the $X_p$ we have $\E \prod_{j=1}^{k} X_{p_{j}} \prod_{j=k+1}^{\ell} \overline{X_{p_{j}}} = 0$. The integral on the left may be calculated explicitly as
$$ \frac{1}{T} \left[ \frac{\exp\bigl\{-it(\sum_{j=1}^{k} \log p_{j} - \sum_{j=k+1}^{\ell} \log p_{j}) \bigr\}}{-i(\sum_{j=1}^{k} \log p_{j} - \sum_{j=k+1}^{\ell} \log p_{j})} \right]_{T}^{2T} \ll \frac{1}{T \Bigl|\log\Bigl(\frac{\prod_{j=1}^{k} p_{j}}{\prod_{j=k+1}^{\ell} p_{j}}\Bigr)\Bigr|} . $$
If $\prod_{j=1}^{k} p_{j} < (3/4)\prod_{j=k+1}^{\ell} p_{j}$ or if $\prod_{j=1}^{k} p_{j} > (4/3)\prod_{j=k+1}^{\ell} p_{j}$ then the logarithmic term here is $> \log 4/3$, so we get an acceptable error term $O(1/T)$. Otherwise, we can write $\log\Bigl(\frac{\prod_{j=1}^{k} p_{j}}{\prod_{j=k+1}^{\ell} p_{j}}\Bigr) = \log\Bigl(1 + \frac{\prod_{j=1}^{k} p_{j} - \prod_{j=k+1}^{\ell} p_{j}}{\prod_{j=k+1}^{\ell} p_{j}}\Bigr)$ and use the Taylor expansion of the logarithm. Since we know that $\prod_{j=1}^{k} p_{j} - \prod_{j=k+1}^{\ell} p_{j} \neq 0$, in fact it is $\geq 1$ and we get a lower bound $\gg \frac{1}{\prod_{j=k+1}^{\ell} p_{j}}$ from the Taylor expansion. Since we are in the case where $\prod_{j=1}^{k} p_{j}$ and $\prod_{j=k+1}^{\ell} p_{j}$ differ at most by a multiplicative factor 4/3, this can also be written as $\gg \frac{1}{\min\{\prod_{j=1}^{k} p_{j} , \prod_{j=k+1}^{\ell} p_{j}\}}$.
\end{proof}

Lemma~\ref{bcorrlem} implies that if we examine the $t$-average of some polynomial expression in the $p^{-it}$, this will be close to the corresponding average of the genuinely random $X_p$ provided that when we expand things out, the product of the primes involved is small compared with $T$. Since one can approximate quite general functions using polynomials (with the degree and coefficient size increasing as one looks for better approximations), one can hope to show rigorously that the distribution of sums of the $p^{-it}$ is often close to the distribution of sums of the $X_p$. A particular instance of this is the well known {\em method of moments} from probability theory. For example, if $P=P(T)$ is some large quantity, $(a_p)_{p \; \text{prime}} = (a_p(T))_{p \; \text{prime}}$ are complex numbers, and if one can show that for each $k \in \N$ one has
$$ \frac{1}{T} \int_{T}^{2T} \Bigl(\Re \sum_{p \leq P} a_{p} p^{-it}\Bigr)^{k} dt \rightarrow \E N(0,1)^k \;\;\; \text{as} \; T \rightarrow \infty , $$
then it follows that the distribution of $\Re \sum_{p \leq P} a_{p} p^{-it}$
converges to the standard Normal distribution as $T \rightarrow \infty$. (Here
we wrote $\E N(0,1)^k = (2\pi)^{-1/2} \int_{-\infty}^{\infty} {w^k e^{-w^{2}/2}} dw$ to denote the $k$-th power moment of the standard Normal distribution.) In view of the above discussion, if the size of the $a_p$ is under control then one could hope to prove such convergence (presuming it actually holds!) when $P(T) = T^{o(1)}$, so that the error terms in Lemma~\ref{bcorrlem} don't contribute too much.

\medskip
Our other basic principle is the following.
\begin{principle}\label{bprin2}
For many purposes (especially statistical questions not directly involving the zeta zeros), for any $\sigma \geq 1/2$ the Riemann zeta function $\zeta(\sigma+it)$ ``behaves like'' an Euler product $\prod_{\text{primes} \; p \leq P} (1 - \frac{1}{p^{\sigma+it}})^{-1}$ of ``suitable'' length $P=P(\sigma, t)$.
\end{principle}

As discussed in the Introduction, the reason for believing that something like Principle~\ref{bprin2} could prevail is that $\zeta(\sigma + it)$ is equal to an Euler product when $\sigma > 1$, and if the primes are well distributed then one expects this identity to continue to influence the behaviour of the zeta function for smaller $\sigma$. Indeed, the Riemann Hypothesis is the statement that it does continue to have an influence, at least to the extent that $\zeta(\sigma + it) \neq 0$ (like a finite product of non-vanishing terms) when $\sigma > 1/2$.

It is much harder to prove rigorous statements corresponding to Principle~\ref{bprin2} than it was for Principle~\ref{bprin1}, and we shall discuss several examples of such statements in the sequel. One also needs to think carefully about the appropriate sense of ``behaves like'' here, especially when $\sigma = 1/2$, since the Riemann zeta function does have infinitely many zeros on the critical line which don't reflect Euler product type behaviour. But to fix ideas a little we state one nice result, which we will also come back to later.

\begin{proposition}[Radziwi\l\l \, and Soundararajan, 2017]\label{radzsoundapprox}
For all $T \leq t \leq 2T$, except for a set whose measure is $o(T)$ as $T \rightarrow \infty$, we have
$$ \zeta\Bigl(1/2 + \frac{W}{\log T} + it\Bigr) = (1+o(1))\exp\Bigl\{\sum_{p^{k} \leq P} \frac{1}{k p^{k(1/2 + W/\log T + it)}} \Bigr\} , $$
where the sum is over prime powers $p^k$. Here $W= (\log\log\log T)^4$, and $P = T^{1/(\log\log\log T)^2}$, and the $o(1)$ term tends to $0$ as $T \rightarrow \infty$.
\end{proposition}

The reader needn't be too concerned about the exact choices of $W$ and $P$ here, and in any event there is some flexibility in those (they are related though, as $W$ increases one can take $P$ smaller). Proposition~\ref{radzsoundapprox} says that $\zeta(s)$ behaves like an Euler product (or the exponential of a prime number sum) provided one shifts away from the critical line $\Re(s) = 1/2$ by a small amount $W/\log T$. As discussed earlier, such a statement cannot hold when $\Re(s) = 1/2$ because of the zeros of the zeta function. But knowing the result when $\Re(s)$ is slightly larger is sufficient for many purposes, since one can use derivative estimates (or more sophisticated statements of a similar character) to pass from knowing things just off the critical line to knowing things on the critical line.

We give a brief sketch of Radziwi\l\l \, and
Soundararajan's~\cite{radsoundclt} proof of
Proposition~\ref{radzsoundapprox}. Using a mean square calculation, one can
show that for most $T \leq t \leq 2T$ we have $\zeta(1/2 + \frac{W}{\log T} +
it) M(1/2 + \frac{W}{\log T} + it) = 1+o(1)$, 
where $M(s) = \sum_{n} \frac{c(n)}{n^s}$ and the coefficients~$c(n)$ are a truncated version of the coefficients one gets by formally expanding the product $\prod_{\text{primes} \; p} (1 - \frac{1}{p^{\sigma+it}})$. Note that one can compute such mean square averages fairly easily using e.g.\ classical approximations $\zeta(s) \approx \sum_{n \leq X} \frac{1}{n^s}$ for the zeta function, provided the coefficients $c(n)$ are zero when $n$ is large (larger than $T^{\epsilon}$, say). Proposition~\ref{radzsoundapprox} follows by combining this with the observation that, for most $T \leq t \leq 2T$, we have $M(1/2 + \frac{W}{\log T} + it) = (1+o(1)) \exp\{-\sum_{p^{k} \leq P} \frac{1}{k p^{k(1/2 + W/\log T + it)}} \}$ (note the minus sign here), which follows from the construction of $c(n)$, the series expansion of the exponential, and (importantly) the fact that $\sum_{p^{k} \leq P} \frac{1}{k p^{k(1/2 + W/\log T + it)}}$ isn't too large for most $T \leq t \leq 2T$.

\medskip
We conclude this section with some small computations that will recur a number of times in the sequel. Firstly we have a number theoretic calculation, which the reader might wish to compare with our earlier discussion of the method of moments.
\begin{lemma}\label{pvarlem}
For any $T \geq 1$ and any $1 \leq x \leq y$, we have
$$  \frac{1}{T} \int_{T}^{2T} \Re \sum_{x \leq p \leq y} \frac{1}{p^{1/2+it}} dt = O\Bigl(\frac{\sqrt{y}}{T}\Bigr) , $$
$$ \frac{1}{T} \int_{T}^{2T} \Bigl(\Re \sum_{x \leq p \leq y} \frac{1}{p^{1/2+it}}\Bigr)^2 dt = \frac{1}{2} \log\Bigl(\frac{\log y}{\log x}\Bigr) + O\Bigl(\frac{1}{\log^{100}(2x)} + \frac{y^2}{T}\Bigr) . $$
\end{lemma}

\begin{proof} 
We rewrite $\Re \sum_{x \leq p \leq y} \frac{1}{p^{1/2+it}} = \frac{1}{2}\bigl(\sum_{x \leq p \leq y} \frac{1}{p^{1/2+it}} + \overline{\sum_{x \leq p \leq y} \frac{1}{p^{1/2+it}}}\bigr)$. The first statement in Lemma~\ref{pvarlem} follows directly by combining this with Lemma~\ref{bcorrlem}. For the second statement, we can rewrite the left hand side as
$$ \frac{1}{4} \sum_{x \leq p,q \leq y} \frac{1}{\sqrt{pq}} \frac{1}{T} \int_{T}^{2T} \bigl(p^{-it}q^{-it} + p^{-it}\overline{q^{-it}} + \overline{p^{-it}}q^{-it} + \overline{p^{-it}}\overline{q^{-it}}\bigr) dt , $$
and using Lemma~\ref{bcorrlem} this is the same as
$$  \frac{1}{4} \sum_{x \leq p,q \leq y} \frac{1}{\sqrt{pq}} \biggl(\E X_{p} X_{q} + \E X_{p} \overline{X_{q}} + \E \overline{X_{p}} X_{q} + \E \overline{X_{p}} \overline{X_{q}} + O\Bigl(\frac{\min\{p,q\}}{T}\Bigr) \biggr) . $$
It is easy to check that if $p \neq q$ then by independence and symmetry all of these expectations vanish, whereas when $p=q$ we have $\E X_{p}^2 = 0$ and $\E |X_p|^2 = 1$. Lemma~\ref{pvarlem} finally follows using the standard estimate $\sum_{x \leq p \leq y} \frac{1}{p} = \log\bigl(\frac{\log y}{\log x}\bigr) + O\bigl(\frac{1}{\log^{100}(2x)}\bigr)$, say, which follows from the Prime Number Theorem.
\end{proof}

Lemma~\ref{pvarlem} tells us that the mean value of $\Re \sum_{p \leq \sqrt{T}} \frac{1}{p^{1/2+it}}$ (say) is very small for large $T$, and the mean square (which is essentially also the variance, since the mean is small) is $\sim (1/2)\log\log T$. We will see the quantity $\log\log T$ appear in many places later, and this variance calculation is one of the key sources of it. Let us emphasise that $\log\log T \rightarrow \infty$ as $T \rightarrow \infty$, whereas if one attempted a similar calculation with $\Re \sum_{p \leq \sqrt{T}} \frac{1}{p^{\sigma+it}}$ for any fixed $\sigma > 1/2$ then the mean square would be convergent. This is one of the key sources of difficulty and interest on the critical line, as compared with elsewhere in the complex plane. On the other hand, $\log\log T$ is a very slowly growing function, which turns out to be key to the success of many of the arguments that we can implement.

We also record a probabilistic calculation.
\begin{lemma}\label{maxnormlem}
Let $Z_1, \dots{}, Z_n$ be independent Gaussian random variables, each having mean zero and standard deviation $\sigma > 0$. Then for any $u \geq 0$, we have
$$ 1 - e^{-\Theta\bigl(n \frac{e^{-u^{2}/2}}{1+u}\bigr)} \leq \p\Bigl(\max_{1 \leq i \leq n} Z_i > u\sigma\Bigr) \ll n \frac{e^{-u^{2}/2}}{1+u} . $$
In particular, for any $\epsilon > 0$ we have
$$ \p\Bigl(1-\epsilon \leq \frac{\max_{1 \leq i \leq n} Z_i}{\sigma\sqrt{2\log n}} \leq 1\Bigr) \rightarrow 1 \;\;\;\;\; \text{as} \; n \rightarrow \infty . $$
\end{lemma}

\begin{proof}[Proof of Lemma~\ref{maxnormlem}]
Using the union bound, we have $\p(\max_{1 \leq i \leq n} Z_i > u\sigma) \leq \sum_{i=1}^{n} \p(Z_i > u\sigma)$. And $\p(Z_i > u\sigma)$ is just the probability that a $N(0,1)$ random variable is $> u$, which is $\asymp \frac{e^{-u^{2}/2}}{1+u}$. This proves the first upper bound.

To prove the lower bound, it will suffice to show that $\p(\max_{1 \leq i \leq n} Z_i \leq u\sigma) \leq e^{-\Theta\bigl(n \frac{e^{-u^{2}/2}}{1+u}\bigr)}$. But by independence, this probability is equal to $\prod_{i=1}^{n} \p(Z_i \leq u\sigma)$. And again, $\p(Z_i \leq u\sigma)$ is $1 - \Theta\bigl(\frac{e^{-u^{2}/2}}{1+u}\bigr) = \exp\bigl\{-\Theta\bigl(\frac{e^{-u^{2}/2}}{1+u}\bigr)\bigr\}$, which gives the result.

For the second statement, we just note that if we take $u = (1-\epsilon)\sqrt{2\log n}$, where $\epsilon > 0$ is small and $n$ is large, then $n \frac{e^{-u^{2}/2}}{1+u} \rightarrow \infty$ as $n \rightarrow \infty$ and so $\p(\max_{1 \leq i \leq n} Z_i > u\sigma) \geq 1 - o(1)$. Similarly, if we take $u = \sqrt{2\log n}$ then $n \frac{e^{-u^{2}/2}}{1+u} \rightarrow 0$ as $n \rightarrow \infty$, and so $\p(\max_{1 \leq i \leq n} Z_i > u\sigma) = o(1)$.
\end{proof}

Note that we made little use of the assumption that the $Z_i$ were Gaussian/Normal random variables. This just gave us a rather explicit form for the tail probabilities $\p(Z_i > u\sigma)$ that arose in the argument. It would also be easy to replace the second statement with something more precise, and later we shall extensively discuss the precise asymptotics of the maxima of Gaussian random variables.

\section{General landscape of the values of zeta}
To set the scene for our discussion of $\zeta(1/2+it)$ in short intervals of $t$, we now review some of the key information we have (both unconditional, conditional and conjectural) when $t$ varies over a wide range.

\medskip
Firstly one might ask about the ``typical'' size of $\zeta(1/2+it)$. A natural way to make this precise is to ask about the distribution of $|\zeta(1/2+it)|$, where $T \leq t \leq 2T$ (say) is chosen uniformly at random. This situation is described by a beautiful classical result of Selberg.
\begin{theorem}[Selberg Central Limit Theorem, 1946]\label{selbergclt}
For any $z \in \R$, we have
$$ \frac{1}{T} \meas\Bigl\{T \leq t \leq 2T : \frac{\log|\zeta(1/2+it)|}{\sqrt{(1/2)\log\log T}} \leq z \Bigr\} \rightarrow \Phi(z) \;\;\; \text{as} \; T \rightarrow \infty , $$
where $\meas\{\cdot\}$ denotes Lebesgue measure, and where $\Phi(z) := \int_{-\infty}^{z} \frac{e^{-w^{2}/2}}{\sqrt{2\pi}} dw$ is the standard Normal cumulative distribution function.
\end{theorem}

Let us remark that although we will have $\zeta(1/2+it) = 0$ (and therefore $\log|\zeta(1/2+it)|$ will be undefined) for some points $T \leq t \leq 2T$ (in fact for $\asymp T\log T$ points), since these points form a discrete set they contribute nothing from the point of view of measure, so are irrelevant to the statement of Theorem~\ref{selbergclt}.

The Selberg Central Limit Theorem is the prototypical manifestation of the Basic Principles discussed in the previous subsection. Looking at things heuristically, we have
$$ \log|\zeta(1/2+it)| = \Re\log\zeta(1/2+it) \approx - \Re \sum_{p \leq P} \log\Bigl(1 - \frac{1}{p^{1/2+it}}\Bigr) \approx \Re \sum_{p \leq P} \frac{1}{p^{1/2+it}} , $$
for ``suitable'' $P = P(T)$. Then as $T \leq t \leq 2T$ varies, the terms $p^{-it}$ behave like independent random variables, and so $\log|\zeta(1/2+it)|$ behaves roughly like a sum of many independent random variables. This is exactly the situation where one expects to have convergence in distribution to a Normal random variable. The second part of the heuristic is rather easy to make rigorous to an acceptable level of precision in this setting, by computing moments of the sums $\Re \sum_{p \leq P} \frac{1}{p^{1/2+it}}$ and showing that they converge to the moments of a Normal distribution. The approximation $\log|\zeta(1/2+it)| \approx \Re \sum_{p \leq P} \frac{1}{p^{1/2+it}}$ has traditionally been more difficult to establish rigorously. As we already discussed, nothing like this can hold pointwise on the critical line because the left hand side will be undefined at some points $t$, so one wants to show that $\log|\zeta(1/2+it)| \approx \Re \sum_{p \leq P} \frac{1}{p^{1/2+it}}$ in some kind of average sense. The classical proofs of this entailed quite complicated manipulations to work around the zeros of zeta, but recently Radziwi\l\l \, and Soundararajan~\cite{radsoundclt} have given a very neat and conceptual proof using Proposition~\ref{radzsoundapprox}.

\medskip
Another key question is about the largest values attained by $|\zeta(1/2+it)|$ as $t$ varies. Unconditionally, our best upper bounds for the size of $\zeta(1/2+it)$ are rather weak despite the application of some very powerful methods to the problem.
\begin{theorem}[Bourgain, 2017]
For any $\epsilon > 0$ and all large $t$, we have the upper bound $|\zeta(1/2+it)| \ll_{\epsilon} t^{13/84 + \epsilon}$ (where the implicit constant may depend on $\epsilon$).
\end{theorem}
Bourgain~\cite{bourgainzeta} proved this result by combining the Hardy--Littlewood approximation $\zeta(1/2+it) \approx \sum_{n \leq t} \frac{1}{n^{1/2+it}}$, exponential sum methods of Bombieri--Iwaniec, and progress in the theory of ``decoupling'' from harmonic analysis. For comparison, general complex analysis arguments (``convexity'') can prove a bound $\ll_{\epsilon} t^{1/4 + \epsilon}$, and long ago Hardy and Littlewood proved the bound $\ll_{\epsilon} t^{1/6 + \epsilon}$. Bourgain's exponent $13/84 \approx 0.155$ is the latest in a long line of improvements. Meanwhile the classical {\em Lindel\"of Hypothesis} (the truth of which follows from the Riemann Hypothesis) conjectures that $|\zeta(1/2+it)| \ll_{\epsilon} t^{\epsilon}$ for any $\epsilon > 0$ and all large $t$.

The bound $t^{\epsilon}$ proposed by the Lindel\"of Hypothesis is still rather soft, so what upper bound should we really expect, in other words what is the true size of $\max_{T \leq t \leq 2T} |\zeta(1/2+it)|$? There isn't a universal consensus about this, but the following results set some limits on where the truth can lie.
\begin{theorem}[Littlewood, 1924]\label{littlewoodupper}
If the Riemann Hypothesis is true, then for all large $t$ we have
$$ |\zeta(1/2+it)| \leq \exp\left\{C\frac{\log t}{\log\log t}\right\} , $$
for a certain absolute constant $C > 0$. 
\end{theorem}

\begin{theorem}[Bondarenko and Seip, 2018]\label{bondseiplower}
For all large $T$, we have
$$ \max_{1 \leq t \leq T} |\zeta(1/2+it)| \geq \exp\left\{(1 + o(1))\sqrt{\frac{\log T \log\log\log T}{\log\log T}}\right\} , $$
where the $o(1)$ term tends to $0$ as $T \rightarrow \infty$.
\end{theorem}

Apart from a sequence of improvements to the value of $C$,
Littlewood's~\cite{littlewood} result in Theorem~\ref{littlewoodupper} hasn't
been improved for almost a century. Theorem~\ref{bondseiplower} is a recent
breakthrough of Bondarenko and Seip~\cite{bondseip, bondseip2}, improving on
earlier lower bounds of a similar shape but without the $\log\log\log T$
factor inside the square root. By further elaboration of their method, the constant $1+o(1)$ has even more recently been improved to $\sqrt{2}+o(1)$ by La Bret\`eche and Tenenbaum~\cite{dlBtenGalSums}.

\medskip
To appreciate these bounds, and contemplate where the truth might lie between them, it is instructive to consider a rough outline of the proofs.

Assuming the truth of the Riemann Hypothesis, one can prove upper bounds of roughly the following shape: for any large $t$ and any parameter $x \leq t$, we have
\begin{equation}\label{soundupper}
\log|\zeta(1/2+it)| \lesssim \Re \sum_{p \leq x} \frac{1}{p^{1/2+it}} + O\left(\frac{\log t}{\log x}\right) .
\end{equation}
See, for example, the Main Proposition of Soundararajan~\cite{soundmoments}. Note that this is another very nice manifestation of Principle~\ref{bprin2}: if we are only interested in upper bounds, we can control $\log|\zeta(1/2+it)|$ by sums over primes at {\em every} point $t$, even on the critical line. As noted previously, one cannot hope for a similar {\em lower} bound at every point, since when $\zeta(1/2+it) = 0$ the left hand side will be undefined (equal to $-\infty$, informally).

It is difficult to give a pointwise bound for this sum over primes except in a trivial way (especially when $x$ is small), namely $\Re \sum_{p \leq x} \frac{1}{p^{1/2+it}} \leq \sum_{p \leq x} \frac{1}{\sqrt{p}} \sim \frac{2x^{1/2}}{\log x}$. So to obtain the best possible upper bound for $\log|\zeta(1/2+it)|$, we choose $x$ to balance the size of this term and the ``big Oh'' term. Choosing $x \asymp \log^{2}t$ is optimal, and yields the claimed bound $\log|\zeta(1/2+it)| \ll \frac{\log t}{\log\log t}$ assuming the Riemann Hypothesis.

To prove their lower bound, Bondarenko and Seip~\cite{bondseip} work to compare the sizes (roughly speaking) of $\int_{1}^{T} \zeta(1/2+it) |R(t)|^2 dt$ and $\int_{1}^{T} |R(t)|^2 dt$, where $R(t)$ is an auxiliary ``resonator'' function that is chosen to concentrate its mass at points where $\zeta(1/2+it)$ should be large. For any choice of $R(t)$, upper bounding $|\zeta(1/2+it)|$ by $\max_{1 \leq t \leq T} |\zeta(1/2+it)|$ implies that
$$ \max_{1 \leq t \leq T} |\zeta(1/2+it)| \geq \frac{|\int_{1}^{T} \zeta(1/2+it) |R(t)|^2 dt|}{\int_{1}^{T} |R(t)|^2 dt} , $$
and if $R(t)$ is well chosen one can hope for this lower bound to be fairly efficient.

One of Bondarenko and Seip's main innovations, as compared with previous arguments, is to choose $R(t) = \sum_{m \in \mathcal{M}} r(m) m^{-it}$ for certain intricately constructed coefficients $r(m)$ whose support is {\em not} constrained to the interval $[1,T]$ (as would be usual to allow one to control error terms when evaluating the integrals). Instead they allow $r(m) \neq 0$ even when $m$ is extremely large, although only on a very sparse sequence of $m$ so that the error terms remain under control. Very roughly speaking, Bondarenko and Seip's resonator $R(t)$ concentrates its mass on those $t$ for which $\sum_{Pe < p < Pe^{(\log\log T)^{c}}} \frac{1}{p^{1/2 + it}}$ is very large, where $P = C\log T \log\log T$ and $c < 1$ and $C$ are suitable constants, and with some penalisation (something like $\frac{1}{\log(p/P)}$) of the larger $p$ in the interval that are harder to control. For a typical $t$, one expects this sum to have order roughly its standard deviation, namely
$$ \approx \sqrt{\sum_{e < \frac{p}{P} <
    e^{(\log\log T)^{c}}} \frac{1}{p \log(\frac{p}{P})} } \asymp \sqrt{\frac{1}{\log\log T} \sum_{e < \frac{n}{P} < e^{(\log\log T)^{c}}} \frac{1}{n \log(\frac{n}{P})} } \asymp \sqrt{\frac{\log\log\log T}{\log\log T} } , $$
using e.g.\ Chebychev's bounds for the density of the primes. So if we look for the largest values attained as $1 \leq t \leq T$ varies, then motivated by the second part of Lemma~\ref{maxnormlem} we could expect this to have size $\asymp \sqrt{\log T} \sqrt{\frac{\log\log\log T}{\log\log T} }$, and this is precisely the lower bound that Bondarenko and Seip are able to prove for zeta by computing the integrals with their choice of $R(t)$. Note that when applying Lemma~\ref{maxnormlem} to get an idea of what to expect here, it is unimportant whether we assume that varying over $1 \leq t \leq T$ corresponds to taking about $T$ independent samples (as would be the usual heuristic, see below), or $T^2$ or $\sqrt{T}$ samples (say), as the logarithms of all these quantities have the same order of magnitude.

\medskip
If we compare these upper and lower bound arguments, we see that both of them come down to analysing contributions from fairly small primes, of size $\log^{2}T$ at most. In the upper bound arguments, one would like to show some cancellation but is forced to resort to trivial estimates, whereas in the lower bounds one wants to show large values {\em are} attained. But considering the problem heuristically, there is no reason to believe that the extreme behaviour of $|\zeta(1/2+it)|$ should be dominated by the behaviour of these very small primes that we are forced to focus on due to methodological limitations.

We have the following conjecture of Farmer--Gonek--Hughes~\cite{farmergonekhughes} about the true size of $\max_{0 \leq t \leq T} |\zeta(1/2+it)|$.
\begin{conjecture}[Farmer, Gonek and Hughes, 2007]\label{fghconj}
We have
$$ \max_{0 \leq t \leq T} |\zeta(1/2+it)| =  \exp\Bigl\{\bigl(\frac{1}{\sqrt{2}} + o(1)\bigr)\sqrt{\log T \log\log T}\Bigr\} , $$
where the $o(1)$ term tends to $0$ as $T \rightarrow \infty$.
\end{conjecture}

Farmer, Gonek and Hughes supply various arguments in support of this conjecture, including a random matrix model, a random primes model (essentially Principle~\ref{bprin1}), and a combination of these. If we assume that something like the Selberg Central Limit Theorem remains valid in a very large deviations regime (so that $\log|\zeta(1/2+it)| \approx N(0,(1/2)\log\log T)$), and further assume that varying over $0 \leq t \leq T$ corresponds to taking about $T$ independent samples, then Lemma~\ref{maxnormlem} would suggest that
$$ \max_{0 \leq t \leq T} \log|\zeta(1/2+it)| \approx \sqrt{2\log T} \sqrt{(1/2)\log\log T} = \sqrt{\log T \log\log T} . $$
So this rather simple approach gives a conjecture of the same shape as Conjecture~\ref{fghconj}, although with a constant 1 instead of $1/\sqrt{2}$ in the exponent. Farmer, Gonek and Hughes~\cite{farmergonekhughes} credit this observation to Montgomery. If one combines this simple line of argument with the bound \eqref{soundupper}, one can in fact recover Conjecture~\ref{fghconj} exactly. Thus if $T \leq t \leq 2T$, and we take $x = e^{\sqrt{\log T}}$ in \eqref{soundupper}, we get (assuming RH)
$$ \log|\zeta(1/2+it)| \lesssim \Re \sum_{p \leq e^{\sqrt{\log T}}} \frac{1}{p^{1/2+it}} + O\bigl(\sqrt{\log T} \bigr) . $$
This choice of $x$ is basically the smallest possible such that the ``big Oh'' term will be of small order compared with the lower bounds we have. (If the reader prefers, he or she could take $x = e^{\sqrt{\log T \log\log T}}$ so the ``big Oh'' term would really be smaller than the lower bound we know from Theorem~\ref{bondseiplower}. This will make no difference to the conjecture we shall derive, it would just be messier to write!) Then using Lemma~\ref{pvarlem}, the mean square of $\Re \sum_{p \leq e^{\sqrt{\log T}}} \frac{1}{p^{1/2+it}}$ is $\sim (1/4)\log\log T$. So now if we assume that this sum will behave like a $N(0,(1/4)\log\log T)$ random variable, then Lemma~\ref{maxnormlem} suggests that
\begin{align*}
  \max_{0 \leq t \leq T} \log|\zeta(1/2+it)| \approx
  & \sqrt{2\log T} \sqrt{\frac{1}{4} \log\log T} +O(\sqrt{\log T}) \\
  =& \sqrt{\frac{1}{2} \log T \log\log T} + O(\sqrt{\log T}) .
\end{align*}

\medskip

The key thing that links all the heuristic arguments leading to Conjecture~\ref{fghconj}, and other conjectures of the same shape, is the assumption of some {\em independence} somewhere. One generally assumes independence in the values of $\zeta(1/2+it_1), \zeta(1/2+it_2)$ when $t_1, t_2$ are sufficiently far apart (e.g.\ $|t_1 - t_2| > 1$, although as noted above the analysis isn't very sensitive to the details of this). One also assumes some independence in modelling the value distribution of zeta at a single point (this is explicit in the random primes model/Principle~\ref{bprin1}, in the random matrix models it is less explicit but there is still much independence in the definition of the random matrices and in their behaviour).

In contrast, if one believes that instead of independence there could be an {\em extreme conspiracy} in the values of the $p^{-it}$ for $p$ small, then one might reasonably believe that the upper bound in Theorem~\ref{littlewoodupper} is closer to the truth. Farmer, Gonek and Hughes~\cite{farmergonekhughes} discuss this in the final section of their paper. The author tends to prefer the independence to the conspiracy assumption, but it is hard to see how one can really distinguish between these possibilities short of actually determining the size of $\max_{0 \leq t \leq T} |\zeta(1/2+it)|$, which we are probably still far from doing.

\section{The conjecture of Fyodorov--Hiary--Keating}
Whereas the Selberg Central Limit Theorem gives, unconditionally, a full description of the typical behaviour of $\log|\zeta(1/2+it)|$ (at least to an initial level of precision), we have seen that our understanding of the largest values attained by $\log|\zeta(1/2+it)|$ is far less complete. Why is this? One answer is that the largest values attained by $\log|\zeta(1/2+it)|$ correspond to very low probability events (i.e.\ sets of $t$ with measure much smaller than $T$), far in the tails of the distribution. Even in a purely probabilistic setting, such problems can present considerable difficulties. For example, the quantitative error terms in probabilistic central limit theorems are often relatively large, so they become much less useful when directly applied to rare events.

Fyodorov and Keating~\cite{fyodkeat} and Fyodorov, Hiary and Keating~\cite{fyodhiarykeat} recently initiated study of a problem that is intermediate between these two regimes (although rather closer to the typical behaviour than the largest values). 
\begin{problem}\label{shortintervalsprob}
As $T \leq t \leq 2T$ varies, how is $\max_{0 \leq h \leq 1} |\zeta(1/2 + it + ih)|$ distributed?
\end{problem}
Note that for some $T \leq t \leq 2T$, the interval $[t,t+1]$ will contain a point $t^{*}$ for which $|\zeta(1/2 + it^{*})| = \max_{T \leq t \leq 2T} |\zeta(1/2 + it)|$, whose size we don't understand well. But since Problem~\ref{shortintervalsprob} is a distributional question, this small subset of $t$ can be ignored (just as in the statement of the Selberg Central Limit Theorem one needn't worry about the zeros of zeta) and one can hope to have a tractable yet interesting question.

\medskip

The next obvious query is whether we should expect the behaviour of the short interval maximum $\max_{0 \leq h \leq 1} |\zeta(1/2 + it + ih)|$ to be much different than the behaviour of $|\zeta(1/2 + it)|$? At first glance, taking the maximum over an interval of bounded length might not be expected to alter things too significantly, in which case the answer to Problem~\ref{shortintervalsprob} might be a result of a similar shape to the Selberg Central Limit Theorem.

Here is a heuristic line of argument that suggests the distributional behaviour of $\max_{0 \leq h \leq 1} |\zeta(1/2 + it + ih)|$ could actually be a lot different than the behaviour at a single point. For the sake of this argument we shall make three temporary assumptions, then later we will examine how reasonable these are.
\begin{itemize}
\item(Assumption 1) The Selberg Central Limit Theorem remains valid even some way into the tails of the probability distribution, in other words the left hand side of Theorem~\ref{selbergclt} is still well approximated by $\Phi(z)$ even when $z$ grows with $T$ ``at a suitable rate''.

\item(Assumption 2) As $T \leq t \leq 2T$ varies, the values $|\zeta(1/2+it+ih_1)|, |\zeta(1/2+it+ih_2)|$ are ``roughly the same'' when $|h_1 - h_2| \leq 1/\log T$.

\item(Assumption 3) As $T \leq t \leq 2T$ varies, the values $|\zeta(1/2+it+ih_1)|, |\zeta(1/2+it+ih_2)|$ behave ``roughly independently'' when $|h_1 - h_2| > 1/\log T$.
\end{itemize}

Much of our analysis will duplicate steps from the proof of Lemma~\ref{maxnormlem}, but we will write it out explicitly for ease of reference and to attain greater precision at some points.

If Assumption 2 is correct, then we have
$$ \max_{0 \leq h \leq 1} |\zeta(1/2 + it + ih)| \approx \max_{1 \leq j \leq \log T} \Bigl|\zeta\Bigl(1/2 + it + i\frac{j}{\log T}\Bigr)\Bigr| . $$

Then for any real $u$, we have the simple union upper bound:
\begin{align*}
& \frac{1}{T} \meas\Bigl\{T \leq t \leq 2T : \max_{1 \leq j \leq \log T} \Bigl|\zeta\Bigl(1/2+it + i\frac{j}{\log T}\Bigr)\Bigr| \geq e^{u} \Bigr\}  \\
& \leq  \sum_{1 \leq j \leq \log T} \frac{1}{T} \meas\Bigl\{T \leq t \leq 2T : \log\Bigl|\zeta\Bigl(1/2+it + i\frac{j}{\log T}\Bigr)\Bigr| \geq u \Bigr\} 
\end{align*}
If Assumption 1 is correct, and if we assume to simplify the writing that $u \geq \sqrt{\log\log T}$, then each summand here will be
$$ \approx \p\biggl(N(0,1) \geq \frac{u}{\sqrt{(1/2)\log\log T}}\biggr) \approx \frac{\sqrt{\log\log T}}{u} e^{-u^{2}/\log\log T} . $$
In particular, if $u = \log\log T - (1/4)\log\log\log T + U$ for some $U \geq 0$ then the right hand side is
$$ \ll \frac{1}{\sqrt{\log\log T}} e^{-(\log\log T - (1/4)\log\log\log T + U)^{2}/\log\log T} \ll \frac{1}{\log T} e^{-2U} e^{-\Theta(U^{2}/\log\log T)} . $$
Summing over $1 \leq j \leq \log T$, we find that if $U$ is large then the sum will be small, in other words we can expect that for most $t$, the maximum $\max_{0 \leq h \leq 1} |\zeta(1/2 + it + ih)|$ has size {\em at most} $e^{\log\log T - (1/4)\log\log\log T + O(1)}$.

For a lower bound, we note that if Assumption 3 is correct then for any $u \in \R$,
\begin{align*}
& \frac{1}{T} \meas\Bigl\{T \leq t \leq 2T : \max_{1 \leq j \leq \log T} \Bigl|\zeta\Bigl(1/2+it + i\frac{j}{\log T}\Bigr)\Bigr| \leq e^{u} \Bigr\}  \\
 \approx & \prod_{1 \leq j \leq \log T} \frac{1}{T} \meas\Bigl\{T \leq t \leq 2T : \log\Bigl|\zeta\Bigl(1/2+it + i\frac{j}{\log T}\Bigr)\Bigr| \leq u\Bigr\} .
\end{align*}
And using Assumption 1 as before to estimate each term in the product, we find the above is
$$ \approx \left(1 - \frac{\sqrt{\log\log T}}{u} e^{-u^{2}/\log\log T} \right)^{\lfloor \log T \rfloor} . $$
In particular, if we take $u = \log\log T - (1/4)\log\log\log T - U$ for some fixed $U \geq 0$ (note that we have $-U$ here, not $U$) then each bracket will be $\approx (1 - \frac{e^{2U}}{\log T})$, and the product of $\lfloor \log T \rfloor$ copies of this will be small. So we can expect that for most $t$, the maximum $\max_{0 \leq h \leq 1} |\zeta(1/2 + it + ih)|$ has size {\em at least} $e^{\log\log T - (1/4)\log\log\log T + O(1)}$ as well.

\medskip
We shall revisit this heuristic argument later, but record a few immediate observations. Firstly, the typical size of the maximum derived above is close to $e^{\log\log T}$, as opposed to the size $e^{\Theta(\sqrt{\log\log T})}$ at a typical point $t$ provided by the Selberg Central Limit Theorem. So, if the above heuristic is roughly accurate, there should be a real difference between these situations. Note, however, that this size is still much smaller than the regime considered in Theorems~\ref{littlewoodupper} and~\ref{bondseiplower}, so we are much less far into the tails of the distribution and can have hopes of a good rigorous analysis of the situation.

Another striking contrast is that in the Selberg Central Limit Theorem, the distribution of $\log|\zeta(1/2+it)|$ is shown to have mean zero and to vary around this on a scale of $\sqrt{\log\log T}$. In our heuristic for the short interval maximum of log zeta, the random variation occurs on a {\em smaller} scale $O(1)$, whilst one has a deterministic main term of size $\sim \log\log T$.

Let us also note that Assumption 3, the independence assumption, was only required for the proof of the lower bound. Thus one might suspect, and it will turn out to be the case, that it should be easier to make our heuristic argument rigorous for the upper bound than for the lower bound.

\medskip
As well as proposing the study of Problem~\ref{shortintervalsprob}, Fyodorov, Hiary and Keating~\cite{fyodhiarykeat, fyodkeat} also made a precise conjecture about the answer.

\begin{conjecture}[Fyodorov--Hiary--Keating, 2012]\label{fhkconj}
For any real function $g(T)$ that tends to infinity with $T$, we have that
$$ \frac{1}{T} \meas\bigl\{0 \leq t \leq T : \bigl|\max_{|h| \leq 1} \log|\zeta(1/2+it+ih)| - (\log\log T - (3/4)\log\log\log T)\bigr| \leq g(T) \bigr\} $$
tends to 1 as $T \rightarrow \infty$.
\end{conjecture}

In fact, Fyodorov--Hiary--Keating make an even more precise conjecture than this, about the {\em distribution} of the difference between $\max_{|h| \leq 1} \log|\zeta(1/2+it+ih)|$ and $(\log\log T - (3/4)\log\log\log T)$. But this seems far beyond anything that is rigorously attackable at present, so we shall not discuss it further here. We also note that the choice of the interval $|h| \leq 1$ is rather arbitrary, and in fact Fyodorov--Hiary--Keating looked primarily at the interval $0 \leq h \leq 2\pi$, which corresponds more naturally with the random matrix setting. But one will have an analogous conjecture and results for any interval of fixed non-zero length.

Fyodorov, Hiary and Keating were led to their conjecture via a two step process, which we shall briefly explain. For given $t$, in order to understand the behaviour of $\max_{|h| \leq 1} \log|\zeta(1/2+it+ih)|$ one might try to compute quantities such as
$$ \int_{|h| \leq 1} e^{2\beta \log|\zeta(1/2+it+ih)|} dh = \int_{t-1}^{t+1} |\zeta(1/2+ iw)|^{2\beta} dw , $$
for varying $\beta > 0$. The idea is that as $\beta$ becomes larger, the size of the integral will be increasingly dominated by the largest values attained by $\log|\zeta(1/2+it+ih)|$. In the language of mathematical physics, this kind of integral is the {\em partition function} associated with $\log|\zeta(1/2+it+ih)|$. Since we are interested in what happens as $t$ varies, we could further try to understand this by computing quantities such as
$$ \int_{0}^{T} \left( \int_{t-1}^{t+1} |\zeta(1/2+ iw)|^{2\beta} dw \right)^{q} dt , $$
where now $q > 0$ is a further parameter. For given $\beta$, if we can understand the size of these integrals for all (or many) $q$ we might hope to get a good understanding of the distribution of $\int_{t-1}^{t+1} |\zeta(1/2+ iw)|^{2\beta} dw$. And in turn, if one can understand this for suitable $\beta$ one might hope to get a good understanding of $\max_{|h| \leq 1} \log|\zeta(1/2+it+ih)|$.

To understand how all these objects might behave, Fyodorov, Hiary and Keating turned to the well known idea that $\zeta(1/2+it)$ behaves like the characteristic polynomial of suitable random matrices. In the random matrix setting, they were able to compute the analogous integrals for a certain range of $q \in \N$ (depending on $\beta$), when $\beta < 1$. Although this amount of information is {\em not} sufficient to rigorously draw conclusions about the maximum, even in the random matrix setting, they noticed that the quantities computed agreed with some analogous integrals arising in statistical mechanics. The Fyodorov--Hiary--Keating conjecture then arises from supposing that characteristic polynomials of random matrices, and further the Riemann zeta function, behave in the way suggested by those statistical mechanics models.

We shall not say more about Fyodorov, Hiary and Keating's motivation for their conjecture, referring the reader instead to the original papers~\cite{fyodhiarykeat, fyodkeat}, which also describe some interesting numerical evidence. We just note that one of the important features of their statistical mechanics problem is a {\em logarithmic correlation structure}, which we shall discuss much further below. We also note that some parts of Fyodorov, Hiary and Keating's conjectures in the random matrix setting, and about the partition function $\int_{t-1}^{t+1} |\zeta(1/2+ iw)|^{2\beta} dw$, have recently been proved using ideas related to those we shall describe here. See the papers~\cite{abbrandmat, argouiradz, paqzeit}, for example.

\medskip
Conjecture~\ref{fhkconj} suggests that our earlier heuristic analysis isn't quite right, but almost, since the first order term $\log\log T$ that we obtained was the same. But this suggestion is a little misleading. As we shall now explain, it is possible to modify the heuristic to give another supporting heuristic for Conjecture~\ref{fhkconj} (and possible to prove some of this rigorously, as we shall come to later), but this requires quite careful thought about our Assumptions 2 and 3.

Recall that we assumed earlier that $|\zeta(1/2+it+ih_1)|, |\zeta(1/2+it+ih_2)|$ are ``roughly the same'' when $|h_1 - h_2| \leq 1/\log T$, and ``roughly independent'' when $|h_1 - h_2| > 1/\log T$. The reason for these starting assumptions is that when $T \leq t \leq 2T$ is large, we have rigorously (the Hardy--Littlewood approximation) that
$$ \zeta(1/2 + it) = \sum_{n \leq T} \frac{1}{n^{1/2+it}} + O\Bigl(\frac{1}{\sqrt{T}}\Bigr) , $$
and we have heuristically (as in Principle~\ref{bprin2}) that
$$ \log|\zeta(1/2+it)| \approx \Re \sum_{p \leq T^{1/3}} \frac{1}{p^{1/2+it}} , $$
say. In both of these expressions, the most rapidly varying terms are of the form $e^{-it\log n}$ with $\log n \asymp \log T$. Thus if $t$ varies by less than $1/\log T$, we can expect the sums not to change much, but if $t$ varies by more one starts to see significant variation. (Another possible justification is that the average spacing between imaginary parts of zeta zeros around $T$ is $\asymp 1/\log T$.)

The assumption that $\zeta(1/2+it)$ doesn't usually change much when $t$ varies by less than $1/\log T$ is actually very reasonable, at least if one replaces $1/\log T$ by something slightly smaller such as $1/\log^{1.01}T$. But if we look at the sum $\Re \sum_{p \leq T^{1/3}} \frac{1}{p^{1/2+it}}$, although it is true that the terms with $p \approx T^{1/3}$ start to vary when $t$ shifts by more than $1/\log T$, the smaller terms in the sum don't change until $t$ shifts by much more. In some situations (e.g.\ if we looked at $\sum_{p \leq T^{1/3}} p^{-it}$), the size of a sum is dominated by the final terms and so this effect wouldn't matter, but in $\sum_{p \leq T^{1/3}} \frac{1}{p^{1/2+it}}$ the contributions from different parts of the sum are typically much more equal. So $|\zeta(1/2+it+ih_1)|, |\zeta(1/2+it+ih_2)|$ will {\em not} behave entirely independently just because $|h_1 - h_2| > 1/\log T$.

\medskip
To explain this more precisely, note that we can decompose
\begin{equation}\label{decompeq}
\Re \sum_{p \leq T^{1/3}} \frac{1}{p^{1/2+it}} = \sum_{0 \leq k \leq \log\log T} \Re \sum_{e^{e^{k-1}} < p \leq \min\{e^{e^{k}}, T^{1/3}\}} \frac{1}{p^{1/2+it}} .
\end{equation}
By Principle~\ref{bprin1}, since the inner sums here involve disjoint sets of primes we expect them to behave independently of one another as $t$ varies. The reason for decomposing into sums on these ranges is that, by Lemma~\ref{pvarlem}, each inner sum has very small mean value and has mean square
$$ \frac{1}{2} \log\Bigl(\frac{e^k}{e^{k-1}}\Bigr) + O\Bigl(\frac{1}{e^{100k}} + \frac{T^{2/3}}{T}\Bigr) = \frac{1}{2} + O\Bigl(\frac{1}{e^{100k}}\Bigr) . $$
In other words, we have split up into pieces whose typical orders of magnitude are comparable. And for all the terms in the $k$-th sum we have $\log p \asymp e^k$, so the scale of $t$ on which this sum doesn't change much is {\em not} $1/\log T$, but the wider scale $1/e^k$.

Another way to capture this phenomenon is to calculate the {\em correlation} between $\Re \sum_{p \leq T^{1/3}} \frac{1}{p^{1/2+it}}$ and $\Re \sum_{p \leq T^{1/3}} \frac{1}{p^{1/2+it+ih}}$. By exactly the same kind of argument as in Lemma~\ref{pvarlem}, one can show that
\begin{multline*}
  \frac{1}{T} \int_{T}^{2T} \Bigl(\Re \sum_{p \leq T^{1/3}} \frac{1}{p^{1/2+it}}\Bigr) \Bigl(\Re \sum_{p \leq T^{1/3}} \frac{1}{p^{1/2+it+ih}}\Bigr) dt \approx\\ \left\{ \begin{array}{ll}
     (1/2)\log\log T & \text{if} \; |h| \leq 1/\log T \\
     (1/2)\log(1/|h|) & \text{if} \; 1/\log T < |h| \leq 1.
\end{array} \right. 
\end{multline*}
Thus if $|h| \leq 1/\log T$, this average is roughly the same size as the mean square of $\Re \sum_{p \leq T^{1/3}} \frac{1}{p^{1/2+it}}$, in other words the sums are almost perfectly correlated and behave in the same way. As $|h|$ increases, so more and more sums at $t$ and $t+h$ in the decomposition \eqref{decompeq} become decoupled, the correlation goes down and the behaviour of $\Re \sum_{p \leq T^{1/3}} \frac{1}{p^{1/2+it}}$ and $\Re \sum_{p \leq T^{1/3}} \frac{1}{p^{1/2+it+ih}}$ becomes increasingly different.

It is a general principle that, given Gaussian random variables with mean zero and equal variances, if they are positively correlated then the maximum is smaller (in a distributional sense) than if they were independent. For example, this follows from a very useful probabilistic result called Slepian's Lemma. This is somewhat intuitive, since positive correlations mean that the random variables tend to be big or small together, so we have fewer ``genuinely independent'' tries at obtaining a very large value. Thus we can see, in quite a soft way, that if $\log|\zeta(1/2+it+ih_1)|, \log|\zeta(1/2+it+ih_2)|$ are positively correlated (rather than independent) when $|h_1 - h_2| > 1/\log T$, then $\max_{0 \leq h \leq 1} |\zeta(1/2 + it + ih)|$ should be smaller than our initial analysis predicted. This is fully consistent with Conjecture~\ref{fhkconj}.

There is no such soft argument for determining exactly  {\em how much smaller} we should expect the maximum to be in the presence of positive correlation, but in recent years the probabilistic tools to do this have become available.

\begin{theorem}[Harper, 2013]\label{harperthm}
Let $(X_p)_{p \; \textup{prime}}$ be a sequence of independent random variables, each distributed uniformly on the complex unit circle. Then with probability tending to 1 as $T \rightarrow \infty$, we have
\begin{align*}
&  \max_{|h| \leq 1}  \Re \sum_{p \leq T} \frac{X_p}{p^{1/2+ih}} \geq \log\log T - 2\log\log\log T - C(\log\log\log T)^{3/4} \\
\text{and}\quad
  &  \max_{|h| \leq 1}  \Re \sum_{p \leq T} \frac{X_p}{p^{1/2+ih}} \leq  \log\log T - (1/4)\log\log\log T + C\sqrt{\log\log\log T} ,
\end{align*}
where $C > 0$ is a certain absolute constant.
\end{theorem}

\begin{theorem}[Arguin, Belius and Harper, 2017]\label{abharpthm}
Let $(X_p)_{p \; \textup{prime}}$ be a sequence of independent random variables, each distributed uniformly on the complex unit circle. Then for any $\epsilon > 0$, with probability tending to 1 as $T \rightarrow \infty$ we have
\begin{align*}
& \max_{|h| \leq 1} \Re \sum_{p \leq T} \frac{X_p}{p^{1/2+ih}} \geq  \log\log T - (3/4 + \epsilon)\log\log\log T  \\
\text{and}\quad &  \max_{|h| \leq 1} \Re \sum_{p \leq T} \frac{X_p}{p^{1/2+ih}} \leq  \log\log T - (3/4 - \epsilon)\log\log\log T .
\end{align*}
\end{theorem}

Strictly speaking, Theorem~\ref{harperthm} is proved for a slightly different sum (with a smooth weight), and both Theorems are proved for slightly different ranges of $h$. But the methods would certainly yield the stated results. Harper~\cite{harperlcz} proves the upper bound in Theorem~\ref{harperthm} using a union bound argument, and proves the lower bound by substituting the logarithmic correlation structure of these sums into general lower bound results for random processes from~\cite{harpergp}. Arguin, Belius and Harper~\cite{abh} prove Theorem~\ref{abharpthm} by working explicitly with a decomposition like \eqref{decompeq}, using methods from the theory of branching random walks. See e.g. Kistler's survey~\cite{kistler} for a description of such methods as applied in many different contexts. Note that the conclusion of Theorem~\ref{abharpthm} exactly agrees, for these randomised prime number sums, with Conjecture~\ref{fhkconj} (although Theorem~\ref{abharpthm} is less precise). So now {\em if} we believe in suitably strong versions of Principle~\ref{bprin1} (so that $\Re \sum_{p \leq T} \frac{1}{p^{1/2+it+ih}}$ behaves like $\Re \sum_{p \leq T} \frac{X_p}{p^{1/2+ih}}$ as $t$ varies) and Principle~\ref{bprin2} (so that $\log|\zeta(1/2+it+ih)|$ is typically close to $\Re \sum_{p \leq T} \frac{1}{p^{1/2+it+ih}}$ as $t$ varies), then we have another strong reason for believing Conjecture~\ref{fhkconj}.

\section{Progress towards the conjecture}
In this section we describe some rigorous theorems about the zeta function that make progress towards Conjecture~\ref{fhkconj}.

\begin{theorem}[Najnudel, 2018]\label{thmnajnudel}
For any real function $g(T)$ that tends to infinity with $T$, we have
$$ \frac{1}{T} \meas\bigl\{0 \leq t \leq T : \max_{|h| \leq 1} \log|\zeta(1/2+it+ih)| \leq \log\log T + g(T) \bigr\} \rightarrow 1 \;\;\; \text{as} \; T \rightarrow \infty . $$

Furthermore, if the Riemann Hypothesis is true then for any $\epsilon > 0$ we have
$$ \frac{1}{T} \meas\bigl\{0 \leq t \leq T : \max_{|h| \leq 1} \log|\zeta(1/2+it+ih)| \geq (1-\epsilon)\log\log T \bigr\} \rightarrow 1 \;\;\; \text{as} \; T \rightarrow \infty . $$
\end{theorem}

\begin{theorem}[Arguin, Belius, Bourgade, Radziwi\l\l, Soundararajan, 2019]\label{thmabbrs}
Najnudel's Theorem is true without the need to assume the Riemann Hypothesis.
\end{theorem}

Before we turn to the proofs of these results, we make a few explanatory remarks. Najnudel's paper~\cite{najnudel} appeared in preprint form on the arXiv in November 2016, and the independent paper of Arguin, Belius, Bourgade, Radziwi\l\l \; and Soundararajan~\cite{abbrs}, which didn't require the assumption of the Riemann Hypothesis, was posted to the arXiv in December 2016. Najnudel proves analogous results (assuming RH) for the imaginary part of $\log\zeta(1/2+it+ih)$ as well. It is possible, but not certain, that some of these could also be made unconditional using the methods of Arguin, Belius, Bourgade, Radziwi\l\l \; and Soundararajan.

\medskip
The upper bounds in Theorems~\ref{thmnajnudel} and~\ref{thmabbrs} are much easier than the lower bounds, and aside from differences in detail are proved in similar ways. Essentially the same argument was also sketched at the end of the introduction to the author's preprint~\cite{harperlcz}. If we looked at a discrete maximum over points $h = j/\log T$ with $|j| \leq \log T$, instead of the maximum over a continuous interval $|h| \leq 1$, we could argue that
\begin{align*}
& \frac{1}{T} \meas\Bigl\{0 \leq t \leq T : \max_{|j| \leq \log T} \log\Bigl|\zeta\Bigl(1/2+it+i\frac{j}{\log T}\Bigr)\Bigr| > \log\log T + g(T) \Bigr\}  \\
& \leq  \sum_{|j| \leq \log T} \frac{1}{T} \meas\Bigl\{0 \leq t \leq T : \log\Bigl|\zeta\Bigl(1/2+it+i\frac{j}{\log T}\Bigr)\Bigr| > \log\log T + g(T) \Bigr\}  \\
& \leq  \sum_{|j| \leq \log T} \frac{1}{T} \int_{0}^{T} \frac{|\zeta(1/2+it)|^2}{e^{2(\log\log T + g(T))}} dt . 
\end{align*}
It is a classical result of Hardy and Littlewood that $\int_{0}^{T} |\zeta(1/2+it)|^2 dt \sim T\log T$ as $T \rightarrow \infty$, so the right hand side is $\ll e^{-2g(T)}$, which indeed tends to $0$ as $T \rightarrow \infty$. To pass from the continuous maximum to the discrete maximum, one can just use classical analytic techniques such as the Sobolev--Gallagher inequality (essentially estimating the average size of the derivative of $\zeta(1/2+it)$). See e.g.\ the paper of Arguin--Belius--Bourgade--Radziwi\l\l--Soundararajan~\cite{abbrs}. Note that this argument is really quite similar to the heuristic one we gave before, with the second moment asymptotic for the zeta function (which is an exponential moment calculation for $\log|\zeta(1/2+it)|$) providing the necessary large deviation estimate for $\log|\zeta(1/2+it)|$. The fact that we don't get the extra subtracted term $-(1/4)\log\log\log T$ in the rigorous argument reflects a standard inefficiency when bounding large deviation probabilities/measures using exponential moments. 

\medskip
To prove the lower bound in Theorem~\ref{thmnajnudel}, Najnudel's main number theoretic input is a striking estimate of the following shape: if the Riemann Hypothesis is true, and if $t$ is large and $1 \leq x \ll t$ is a parameter, then
\begin{equation}\label{najlower}
\max_{|h| \leq 1} \log|\zeta(1/2+it+ih)| \gtrsim \max_{|h| \leq 1/2} \Re \sum_{p \leq x} \frac{1}{p^{1/2+it+ih}} + O\Bigl(\frac{\log t}{(\log x)^{C}} + \frac{x^C}{t} \Bigr) .
\end{equation}
The reader should compare this with Soundararajan's upper bound \eqref{soundupper}. The correct statement of this lower bound is a bit more complicated, in particular the sum $\sum_{p \leq x} \frac{1}{p^{1/2+it+ih}}$ should really be an infinite sum with a smooth cutoff that decays when $p > x$, and there is some contribution from prime squares as well. But to get an idea of the argument one can just think of \eqref{najlower}.

As in many similar situations (e.g.\ Soundararajan's~\cite{soundmoments} proof of \eqref{soundupper}), Najnudel assumes the Riemann Hypothesis when proving \eqref{najlower} to avoid the appearance of other large terms corresponding to possible zeros of the zeta function off the critical line. This reflects the general duality between prime numbers being well distributed, Euler product type formulae roughly holding, and the zeros of the zeta function being well behaved, as discussed at the very beginning of this paper. The other important thing to note here is the role played by the maximum over $h$. We have remarked several times that it would be impossible to prove a pointwise lower bound comparable to \eqref{soundupper} or \eqref{najlower}, because at a zero of the zeta function the prime number sum is finite but log zeta becomes undefined. Roughly speaking, in the course of proving \eqref{najlower} Najnudel exploits the fact that
$$ \max_{|h| \leq 1/\log^{0.99}x} \log|\zeta(1/2+it+ih)| \geq \frac{\log^{0.99}x}{2} \int_{|h| \leq 1/\log^{0.99}x} \log|\zeta(1/2+it+ih)| dh . $$
On the one hand, one can cover the interval $|h| \leq 1$ by small intervals of length $2/\log^{0.99}x$ (with a small error at the ends, hence the change to the interval $|h| \leq 1/2$ on the right hand side of \eqref{najlower}), and hope that replacing the maximum in each small interval by its average (whilst still taking the maximum {\em over} all the intervals) won't reduce the size too much. On the other hand, since an interval of length $2/\log^{0.99}x$ is large compared with the average spacing $\asymp 1/\log t$ of zeta zeros with imaginary part around $t$, by integrating over such an interval one smooths out (and removes the effect of) the blow-up at the zeros.

The inequality \eqref{najlower} is the manifestation of Principle~\ref{bprin2} in Najnudel's argument. Having passed to prime number sums, with some flexibility in the choice of the length $x$, Najnudel shows that they behave like sums of independent random variables (Principle~\ref{bprin1}) by moment calculations, similarly as discussed following Lemma~\ref{bcorrlem}. Thus he can argue about the size of $\max_{|h| \leq 1/2} \Re \sum_{p \leq x} \frac{1}{p^{1/2+it+ih}}$ with a similar style of argument, motivated by branching random walk, as Arguin, Belius and Harper~\cite{abh} used for their randomised model of zeta.

\medskip
For their unconditional lower bound, Arguin, Belius, Bourgade, Radziwi\l\l \; and Soundararajan use a result like Proposition~\ref{radzsoundapprox} to serve as their realisation of Principle~\ref{bprin2}. The choices of $W$ and $P$ are a bit different than in Proposition~\ref{radzsoundapprox}, but the proof is essentially the same as the one we sketched for that proposition. To apply this to give lower bounds for $\max_{|h| \leq 1} \log|\zeta(1/2+it+ih)|$, a couple of other auxiliary manoeuvres are required. Since Proposition~\ref{radzsoundapprox} concerns points slightly off the critical line, one wants to know that for most $t$, if there is a large value slightly off the critical line there will also be one nearby on the critical line. This is swiftly proved using, essentially, an average bound for the size of the derivative of zeta, obtained by manipulating the Hardy--Littlewood approximation $\zeta(1/2 + it) = \sum_{n \leq T} \frac{1}{n^{1/2+it}} + O(\frac{1}{\sqrt{T}})$. Also, whereas Proposition~\ref{radzsoundapprox} supplies information at most individual points $T \leq t \leq 2T$, Arguin--Belius--Bourgade--Radziwi\l\l--Soundararajan need results that hold for most {\em intervals} $[t-1,t+1]$. This extension is obtained by noting that in the proof of Proposition~\ref{radzsoundapprox}, the individual steps (such as the approximation $\zeta(s) M(s) = 1+o(1)$) hold uniformly for most intervals $[t-1,t+1]$, thanks again to classical Sobolev--Gallagher type manipulations.

By shifting a little off the critical line, and only seeking to approximate (the shifted version of) $\max_{|h| \leq 1} \log|\zeta(1/2+it+ih)|$ by prime number sums for {\em most} $t$, Arguin--Belius--Bourgade--Radziwi\l\l--Soundararajan can avoid Najnudel's appeal to the Riemann Hypothesis.

Having reached this stage, moment calculations with the prime number sums again show that they behave like sums of independent random variables (Principle~\ref{bprin1}), and one can conclude with a branching random walk style argument.

\medskip
We finish with a glance at what remains to be done to prove Conjecture~\ref{fhkconj}. Both Theorems~\ref{thmnajnudel} and~\ref{thmabbrs} are less precise than the conjecture, but it seems quite reasonable to think that the methods have not yet been fully perfected, so that more precise results could be extracted. On the other hand, to increase the precision in these methods one needs to approximate the zeta function by prime number sums that are longer, and at points that are closer to the critical line. At a certain point the influence of the zeta zeros, and (more technically) of off-diagonal terms that would start to appear in the analysis, obstructs progress. Because the scale $\log\log T$ on which one is working grows so slowly with $T$, one has quite a lot of flexibility in truncating sums, etc. if one just wants to get close to the answer, but this starts to disappear if one wants a precise answer.

One particular landmark en route to proving Conjecture~\ref{fhkconj}, which might be achievable, would be to prove that usually $\max_{|h| \leq 1} \log|\zeta(1/2+it+ih)| \leq \log\log T - c\log\log\log T$ for some $c > 1/4$. The Conjecture predicts that one can take $c = 3/4 + o(1)$, whereas we have seen (in our initial heuristic argument) that one would get $c = 1/4 + o(1)$ if $|\zeta(1/2+it+ih_1)|, |\zeta(1/2+it+ih_2)|$ behaved ``roughly independently'' when $|h_1 - h_2| > 1/\log T$. We saw in our later analysis that this shouldn't really be the case, and proving an upper bound with some fixed $c > 1/4$ would give a concrete (if rather subtle) manifestation of this failure of independence.

\medskip
{\em Acknowledgements.} The author would like to thank Louis-Pierre Arguin, Paul Bourgade, and N.~Bourbaki for their comments and suggestions on a draft of this paper.


\end{document}